\documentclass[11pt]{article}
\usepackage{amssymb}
\usepackage{amsmath}
\usepackage{color}

\newenvironment{proof}{\noindent {\bf Proof }}
{\hfill $\bullet$ \vspace{0.25cm}}
\newcommand{\nn}{\nonumber}
\newcommand{\dis}{\displaystyle}

\newtheorem{thm}{Theorem}
\newtheorem{prop}{\indent Proposition}

\newtheorem{lem}{\indent Lemma}

\newcommand{\mmmintone}[1]{{\dis{\int\kern -.36cm-}}_{\kern-.21cm\substack{#1}}\;\;}
\newcommand{\mmmintwo}[2]{{\dis{\int\kern -.43cm-}}_{\kern-.21cm\substack{#1}}^{\substack{#2}}\;\;}
\newcommand{\submint}{{\scriptstyle{\int\kern -.66em -}}}
\newcommand{\submintone}[1]{{\scriptstyle{\int\kern -.66em-}}_{\scriptscriptstyle{\kern-.21em\substack{#1}}}}
\newcommand{\fracmint}{{\textstyle{\int\kern -.88em -}}}
\newcommand{\fracmintone}[1]{{\textstyle{\int\kern -.88em
-}}_{\scriptscriptstyle{\kern-.21em\substack{#1}}}\;}

\title{A free boundary problem in biological selection models}

\author{Jimyeong Lee\footnote{  E-mail: ljm9667@gmail.com } \\Gran Sasso Science Institute, Via. F. Crispi 7, 67100 L'Aquila, Italy}

\date{\today}

\begin{document}

\maketitle

\begin{abstract}
We prove local existence  for classical solutions of a free boundary problem which arises in one of the biological selection models proposed by Brunet and Derrida, \cite{BD}. The problem we consider describes the limit evolution of branching brownian particles on the line with death of the leftmost particle at each creation time as studied in \cite{DFPS}.  We use extensively results in \cite{Cannon} and \cite{Fasano}.
\end{abstract}
%
%{\it Key words} : \\
%
%{\it AMS Classification}  :

\vskip 1cm

%
%
%\begin{abstract}
%
%\end{abstract}

%{\it Key words} : Hydrodynamic limit.\\

%{\it AMS Classification}  : %60F17; 60K35; 60J25

\section{Introduction}

\label{sec:intro}

Brunet and Derrida \cite{BD}, have proposed a class of models for biological selection processes including the one we consider here.  This is a system of $N$ brownian particles on the line which branch independently at rate 1 creating a new brownian particle on the same position of the father; simultaneously the leftmost particle (which is the less fit) disappears.  Thus the total number of particles does not change.

In \cite{DFPS} it is proved that the particles density has a limit as $N$ diverges (under suitable assumptions on the initial datum).  It is also proved that if the following FBP has a ``classical solution'' then this is the same as the limit density of the branching brownians, similar results have been obtained for other FBP, see \cite{CDGP}.

Let $b\in \mathbb R$, say $b>0$, and  $\displaystyle \rho_0\in C_c^{2}([b,\infty))$  such that  $\rho_0(b)=0$, $\frac{d}{dx}\rho_{0}(b)=2$, $\int_{b}^{\infty}\rho_0(x)dx=1$. The free boundary problem FBP $(*)$ that we consider is: fixed $b$ and $\rho_0$ as above
find a continuous curve $L_t$, $t\ge 0$, starting from $L_0 = b$, and $\rho(x,t)$,
$x\ge L_t; t\ge 0$,
 so that
\[
(*)\begin{cases}

     \displaystyle \rho_t(x,t)=\frac{1}{2}\rho_{xx}(x,t)+\rho(x,t),     & \quad \text{if } L_t<x,\ \ t> 0,\\

      \displaystyle \rho(L_t,t)=0,     & \quad \text{if } t \ge 0,\\
     \displaystyle \rho(x,0)=\rho_0(x),  & \quad \text{if } b\leq x,\\
    \displaystyle \int_{L_t}^{\infty}\rho(x,t)dx=\int_{b}^{\infty}\rho_0(x)dx=1,  & \quad \text{if } t> 0.\\

  \end{cases}
\]
\noindent
Classical solutions of the FBP $(*)$ are defined in the next section.
 In this paper we prove the local in time existence of a classical solution of
 the FBP $(*)$, observing that uniqueness follows from the results in \cite{DFPS}. Traveling wave solutions and the fine asymptotics of $L_t$ are studied  in \cite{BBD} for a
large class of FBP including our FBP $(*)$.

 \vskip2cm

\section{Main results}

By classical solutions we mean the following:

\medskip

{\bf Definition.}
The pair $(L,\rho)$ is a classical solution of the FBP $(*)$
in the time interval $[0,T]$, $T>0$,
 if:

 \begin{itemize}

 \item
 $\displaystyle L\in C^{1}([0,T])$, $L_0=b$

  \item
  $\rho\in C(\overline{D_{L,T}})\cap C^{2,1}(D_{L,T})$, where $\displaystyle D_{L,T}=\{(x,t): L_t<x,\ 0<t<T\}$

  \item   $(L,\rho)$ satisfies the equations in  $(*)$.
\end{itemize}

 \medskip

\begin{thm}
\label{34}

There is $T>0$ so that the FBP $(*)$ has a classical solution
in the time interval $[0,T]$  (denoted by $(L,\rho)$).

%Let $\displaystyle \rho_0\in C_c^{2}([b,\infty))$ be such that  $\rho_0(b)=0$, $\rho_{0}^{\prime}(b)=2$ and $\int_{b}^{\infty}\rho_0(x)dx=1$ for some $b\in\mathbb{R}$. Then there exists $T>0$ so that $(*)$ has a classical solution for $t\in[0,T]$. Namely there is a curve $L=\{L_t,t\in[0,T]\}$, $L_0=b$ and  $\displaystyle L\in C^{1}([0,T])$ and a function $\rho\in C(\overline{D_{L,T}})\cap C^{2,1}(D_{L,T})$, where $\displaystyle D_{L,T}=\{(x,t): L_t<x,\ 0<t<T\}$, which satisfies $(*)$.
\end{thm}

\medskip

We will prove Theorem \ref{34} in the next sections.  We will also prove that
the derivative $\rho_x(x,t)$  has a limit when $x\rightarrow L_t$ denoted by $\displaystyle\rho_x(L_t,t)$ and that it vanishes as $x\to \infty$.
 This together with the conservation law implies  $\displaystyle\rho_x(L_t,t)=2$ for all $t\in[0,T]$. Furthermore if also $\rho_{xx}(L_t,t)$ exists, by differentiating $\rho(L_t,t)=0$ we get $\displaystyle  \dot {L_t}=-\frac{1}{4}\rho_{xx}(L_t,t)$.\\

Given $(L,\rho)$ as above we define $\displaystyle v(x,t):=e^{-t}\rho_{x}(x,t)$ and observe that $(L,v)$  satisfies the following free boundary problem FBP $(**)$:
\[
(**)\begin{cases}

      \displaystyle v_t(x,t)=\frac{1}{2}v_{xx}(x,t),  & \quad \text{if } L_t<x,\ \ t> 0,\\

         \displaystyle v(L_t,t)=2e^{-t},  & \quad \text{if } t> 0,\\

      \displaystyle  v(x,0)=h(x),  &\quad  h\in C_c^{1}([b,\infty)) \\

     \displaystyle  \dot {L_t}=-\frac{1}{4}e^{t}v_x(L_t,t),  & \quad \text{if }  t>0,
  \end{cases}
\]

% where $h\in C_c^{1}([b,\infty))$ ($h$ being the space derivative of $\rho_0$).

 To prove  Theorem \ref{34} we will first prove the existence of a classical solution of the FBP $(**)$  in $[0,T]$:

 \medskip

\begin{thm}
\label{3424}
 There is $T>0$ and a pair $(L,v)$ which satisfies the FBP $(**)$ in $[0,T]$ with:   $\displaystyle L\in C^{1}([0,T])$, $L_0=b$, and $v\in C^{2,1}(D_{L,T})$, where $\displaystyle D_{L,T}=\{(x,t): L_t<x,\ 0<t\leq T\}$.
\end{thm}

\medskip

We did not find a proof of Theorem \ref{3424} in the existing literature, see for instance \cite{BH}, \cite{CV} and references therein.  Our proof exploits the one dimensionality of the problem and uses extensively the Cannon estimates, \cite{Cannon}, following the strategy proposed by Fasano in \cite{Fasano}.
In the next sections we will prove Theorem \ref{3424} and then, taking $h$ equal to the space derivative of $\rho_0$, we will prove Theorem \ref  {34} as a corollary  of Theorem \ref{3424}.

\section{Strategy of the proof of  Theorem \ref{3424} }
The idea is to reduce the analysis of the FBP $(**)$ to a fixed point problem:

\begin{itemize}

\item Take a curve $L_t, t\in[0,T]$ and find $v$ which solves from the first equation to the third one of $(**)$ in the domain $[L_t,\infty)$, $t\in[0,T]$

\item Construct a new curve $\displaystyle K[L](t)=b-\frac{1}{4}\int_{0}^{t}e^{\tau}v_x(L_\tau,\tau)d\tau$ for $0\leq t\leq T$

\item Find $L$ so that $K[L]=L$ and prove that the corresponding pair $(L,v)$
solves the FBP (**)

\end{itemize}

 The first task is to prove existence and smoothness of $v$, that we do in this section using   the   lemmas below, see \cite{Cannon}.
Let $A>0$ and
   $$
\Sigma(A,T):=\left\{L\in C[0,T] : L_0=b,\ \left |\frac{L_{t_2}-L_{t_1}}{t_2-t_1}\right |\leq A\quad \text{for}\ 0\leq t_1<t_2\leq T    \right\}.
   $$
%
%
%
%In the proof of  Proposition \ref{96} we use the following lemmas, see \cite{Cannon}. %Let $s_1$ and $s_2$ be
%%two continuous curves such that $s_1(t)<s_2(t), t\in[0,T]$, and let $\displaystyle E_T:=\{(x,t): s_{1}(t)<x<s_{2}(t),\ 0<t\leq T\}$, $\displaystyle B_T:=\{(s_{i}(t),t): 0\leq t\leq T,\ i\in\{1,2\}\}\cup\{(x,0): s_{1}(0)< x< s_{2}(0)\}$.
%
%\medskip
%%\begin{lem}[\emph{The Weak Maximum, Minimum Principle}]
%%Let $u$ solve  $\displaystyle u_t=\frac{1}{2}u_{xx}$ in $E_T$, and suppose that  $u$ is continuous in $E_T\cup B_T$, then
%%\begin{eqnarray}
%%\displaystyle \max_{E_T\cup B_T}u= \max_{B_T}u,\ \  \min_{E_T\cup B_T}u= \min_{B_T}u.
%%\end{eqnarray}
%
%
%%\end{lem}

\medskip

We define $\displaystyle C_{(1)}^{0}((0,T])$ as the subspace of $C((0,T])$ that consists of those functions $\varphi$ such that
\begin{eqnarray}
\displaystyle\lVert\varphi\rVert_{T}^{(1)}=\sup_{0<t\leq T}|\varphi(t)|<\infty.
\end{eqnarray}
Thus $\displaystyle C_{(1)}^{0}((0,T])$ is a Banach space under the norm $\displaystyle\lVert\cdot\rVert_{T}^{(1)}$. \\Let us denote by $G(x,t;\xi,\tau)$ the Gaussian kernel $\displaystyle\frac{1}{\sqrt{2\pi(t-\tau)}}\exp\left\{-\frac{|x-\xi|^{2}}{2(t-\tau)}\right\}$ which satisfies $\displaystyle G_t=\frac{1}{2}G_{xx}$ and $\displaystyle G_\tau=-\frac{1}{2}G_{\xi\xi}$.

\begin{lem}[\emph{jump relation}]
\label{45}
Let $L\in \Sigma(A,T)$ and $\displaystyle\varphi\in C_{(1)}^{0}((0,T])$.  Define
  $$
w_\varphi(x,t;L):=\int_{0}^{t}G(x,t;L_\tau,\tau)\varphi(\tau)d\tau.
  $$
Then
\begin{eqnarray}
\lim_{x\rightarrow L_t^{+}}\frac{\partial w_\varphi}{\partial x}(x,t;L)=-\varphi(t)+\int_{0}^{t}G_x(L_t,t;L_\tau,\tau)\varphi(\tau)d\tau.
\end{eqnarray}

\end{lem}
\begin{proof}
See Lemma 14.2.5. of \cite{Cannon}.
\end{proof}

\medskip

\begin{lem}
\label{3251}
Let $L\in \Sigma(A,T)$, then there exists a unique $\displaystyle\varphi\in C_{(1)}^{0}((0,T])$  such that\\ $\displaystyle\int_{b}^{\infty}h(\xi)G(L_t,t;\xi,0)d\xi-\varphi(t)+\int_{0}^{t}G_x(L_t,t;L_\tau,\tau)\varphi(\tau)d\tau=2e^{-t}$ for all $t\in(0,T]$.
\end{lem}
\begin{proof}
Since $\displaystyle|G_x(L_t,t;L_\tau,\tau)|\leq\frac{C}{\sqrt{t-\tau}}$ and $\displaystyle\int_{b}^{\infty}h(\xi)G(L_t,t;\xi,0)d\xi\in C_{(1)}^{0}((0,T])$, we have a Volterra integral equation:
\begin{eqnarray}
\label{::}
\varphi(t)=\psi(t)+\int_{0}^{t}G_x(L_t,t;L_\tau,\tau)\varphi(\tau)d\tau
\end{eqnarray}
where $\displaystyle\psi(t)=-2e^{-t}+\int_{b}^{\infty}h(\xi)G(L_t,t;\xi,0)d\xi\in C_{(1)}^{0}((0,T])$.
As in \cite{Cannon} p.247 the right hand side of $\ref{::}$ defines a contraction map from $C_{(1)}^{0}((0,T])$ into itself for $T$ small. Thus there is a unique $\displaystyle\varphi\in C_{(1)}^{0}((0,T])$ which satisfies $(\ref{::})$.
\end{proof}
%Before going to the proof, let us define $\displaystyle C_{(\delta)}^{0}\left((0,T]\right)$, $0<\delta\leq 1$, as the subspace of $C((0,T])$ that consists of those functions $\varphi$ such that
%$$\displaystyle \lVert \varphi\rVert_{T}^{(\delta)}=\sup_{0<t\leq T}t^{1-\delta}|\varphi(t)|<\infty.$$
%Then $\displaystyle C_{(\delta)}^{0}\left((0,T]\right)$ is a Banach space under the norm $\displaystyle \lVert \cdot \rVert_{T}^{(\delta)}$.\\ \\
%Let us first prove uniqueness. Suppose that there are two solutions
 %$v_1, v_2$ , then $v_1-v_2$ satisfies the heat equation with the initial condition $0$ and $(v_1-v_2)(L_t,t)=0$. Moreover, we have $\displaystyle \lim_{x\rightarrow \infty}\sup_{0\leq t\leq T}|(v_1-v_2)(x,t)|=0$ so that $v_1=v_2$ by the weak maximum, minimum principle.
\medskip

 \begin{prop}
\label{96}
For any $L\in\Sigma(A,T)$ let $\varphi$ as in Lemma \ref{3251}.  Define
 \begin{equation}
 \label{???}
 v(x,t):=\int_{b}^{\infty}h(\xi)G(x,t;\xi,0)d\xi
+\int_{0}^{t}G_x(x,t;L_\tau,\tau)\varphi(\tau)d\tau
  \end{equation}
Then $v\in C^{2,1}(D_{L,T})$ satisfies from the first equation to the third one of $(**)$ and\\ $\displaystyle \lim_{x\rightarrow \infty}\sup_{0\leq t\leq T}|v(x,t)|=0$.  In addition, $v$ has the right derivative at the boundary $v_x(L_t,t)\in C_{(1)}^{0}((0,T])$.

\end{prop}

\medskip
\begin{proof} \\
By \eqref{???} $v$ satisfies the heat equation with initial datum $h$ and  boundary conditions $\displaystyle v(L_t,t)=2e^{-t}$, $\displaystyle \lim_{x\rightarrow \infty}\sup_{0\leq t\leq T}|v(x,t)|=0$. Let us choose $\displaystyle R>\sup_{t\in [0,T]} L_t$.
%By  Lemma \ref{45}, we obtain
  %\begin{eqnarray}
 %\displaystyle v_x(L_t,t)=\int_{b}^{\infty}(h(\xi)-2)G_x(L_t,t;\xi,0)d\xi-\varphi(t)+\int_{0}^{t}G_x(L_t,t;L_\tau,\tau)\varphi(\tau)d\tau
 %\end{eqnarray}
%so that $\displaystyle v_x(L_t,t)\in C((0,T])$ and $\displaystyle \sup_{0<t\leq T}\left|v_x(L_t,t)\right|<\infty$.
Then we have $v(R,t)\in C_{(1)}^{0}((0,T])$ such that by Lemma 14.4.4. and Theorem 14.4.1. of \cite{Cannon}, we also obtain
 $v_x(L_t,t)\in C_{(1)}^{0}((0,T])$.

 \end{proof}

\section{Proof of  Theorem \ref{3424}}
Theorem \ref{3424} is proved at the end of the section.

\medskip

\begin{lem}
\label{907}
Let $L\in\Sigma(A,T)$ and $v$ as in \eqref{???}, then there are %independent
positive constants $C_1, C_2, C_3$ such that
\begin{eqnarray}
\displaystyle \left|v_x(L_t,t)\right|\leq C_{1}\int_{0}^{t}\frac{\left|v_x(L_\tau,\tau)\right|}{\sqrt{t-\tau}}d\tau+(C_{2}\sqrt{t}+C_{3}).
\end{eqnarray}
Moreover, we have $v_x(L_t,t)\in C([0,T])$.
\end{lem}

\medskip

\begin{proof}
\\
 Let us fix $(x,t)\in D_{L,T}$ and let us define $\displaystyle D_{\epsilon,R}^{(t)}:=\{(\xi,\tau): L_\tau+\epsilon<\xi<R,\ \epsilon<\tau<t-\epsilon\}$ for each $\epsilon>0$ and $R\in\mathbb{R}$. By the Green's identity, we have
\begin{eqnarray}
 \displaystyle \frac{1}{2}(v_\xi G-vG_\xi)_\xi-(vG)_{\tau}=0\ \Longrightarrow\  \ \oint_{\partial D_{\epsilon,R}^{(t)}} \frac{1}{2}(v_\xi G-vG_\xi)d\tau+(vG)d\xi=0.
\end{eqnarray}
 Hence we obtain another representation of $v$ by letting $\epsilon\rightarrow 0$, $R\rightarrow \infty$,
\begin{eqnarray}
\label{777}
&& v(x,t)=\int_{b}^{\infty}h(\xi)G(x,t;\xi,0)d\xi-
\frac{1}{2}\int_{0}^{t}G(x,t;L_\tau,\tau)
v_\xi(L_\tau,\tau)d\tau\nn\\&& \hskip2cm +\int_{0}^{t}e^{-\tau}G_\xi(x,t;L_\tau,\tau)d\tau.
\end{eqnarray}
Differentiating both sides of $(\ref{777})$ with respect to x  and  integrating by parts, we get
\begin{eqnarray}
&& v_x(x,t)=\int_{b}^{\infty}h_{\xi}(\xi)G(x,t;\xi,0)d\xi-
\frac{1}{2}\int_{0}^{t}G_x(x,t;L_\tau,\tau)v_x(L_\tau,\tau)d\tau\nn \\
&&\hskip2cm +2\int_{0}^{t}e^{-\tau}G(x,t;L_\tau, \tau)d\tau,
\end{eqnarray}

\begin{eqnarray}
\label{3023}&& \frac{1}{2}v_x(L_t,t)=\int_{b}^{\infty}h_{\xi}(\xi)
G(L_t,t;\xi,0)d\xi-\frac{1}{2}
\int_{0}^{t}G_x(L_t,t;L_\tau,\tau)v_x(L_\tau,\tau)d\tau\nn\\&&\hskip2cm
+2\int_{0}^{t}e^{-\tau}G(L_t,t;L_\tau, \tau)d\tau.
\end{eqnarray}
Since $\displaystyle \left|G_x(L_t,t;L_\tau,\tau)\right|\leq \frac{C}{\sqrt{t-\tau}}$ and there are constants $C_2$ and $C_3$ such that
\begin{eqnarray}
\left|2\int_{0}^{t}e^{-\tau}G(x,t;L_\tau, \tau)d\tau\right|\leq C_2\sqrt{t},\ \ \left|\int_{b}^{\infty}h_{\xi}(\xi)G(x,t;\xi,0)d\xi\right|\leq\lVert h_{\xi}\rVert_{\infty}\leq C_3,
\end{eqnarray}
then
\begin{eqnarray}
\label{345}
\displaystyle \left|v_x(L_t,t)\right|\leq (C_{2}\sqrt{t}+C_{3})+C_{1}\int_{0}^{t}\frac{\left|v_x(L_\tau,\tau)\right|}{\sqrt{t-\tau}}d\tau.
\end{eqnarray}
Using (\ref{3023}), we have
\begin{eqnarray}
\label{40303}
v_x(L_t,t)\in C([0,T]).
\end{eqnarray}

\end{proof}

\medskip

\begin{lem}
\label{567}

Let $A>\dfrac{C_3}{4}$,   $C_3$ as in  Lemma \ref{907}. Then for all sufficiently small $T$,   $\displaystyle K[L](t):=b-\frac{1}{4}\int_{0}^{t}e^{\tau}v_x(L_\tau,\tau)d\tau$, $0\leq t\leq T$, maps  $K:\Sigma(A,T)\longrightarrow\Sigma(A,T)$.

   \end{lem}

   \medskip
\begin{proof}
By  Lemma 17.7.1. in \cite{Cannon} applied to (\ref{345}),  we obtain
\begin{eqnarray}
\label{09}
\displaystyle \left|v_x(L_t,t)\right|\leq [1+2C_1\sqrt{T}]\exp\{\pi C_1^{2}T\}(C_2\sqrt{t}+C_3).
\end{eqnarray}
Then $\displaystyle\left|\frac{d}{dt}K[L](t)\right|=\frac{1}{4}\left|e^{t}v_x(L_t,t)\right|\leq \frac{1}{4}e^{T} [1+2C_1\sqrt{T}]\exp\{\pi C_1^{2}T\}(C_2\sqrt{T}+C_3)\leq A$ for all sufficiently small $T$ so that the map $K$ is well defined.
\end{proof}
\\

\medskip

 Hereafter   $A>\dfrac{C_3}{4}$ is fixed. We will show that $K$ is continuous, then, since $\Sigma(A,T)$ is convex and compact we can apply the Schauder fixed point theorem to conclude that $K$ has a fixed point.

 \medskip

\begin{lem}
\label{637}
For any sufficiently small $T>0$ the map $K$ defined in  Lemma \ref{567} is continuous on $\Sigma(A,T)$ with sup norm.
\end{lem}

\medskip

\begin{proof}
\\Let $T>0$ be such that $0<b-AT$ and so small that $K$ is well defined. For $L\in\Sigma(A,T)$, we have $0<b-AT\leq\inf L\leq \sup L\leq b+AT$. Let $v$ be the function determined by $L$ via Proposition \ref{96}. Then $v$ satisfies Green's identity as follows; for $\displaystyle D_{\epsilon,R}^{(t)}$ defined in the proof of  Lemma \ref{907},
\begin{eqnarray}
\label{34512}\oint_{\partial D_{\epsilon,R}^{(t)}}[\xi vd\xi+\frac{1}{2}(\xi v_\xi-v)d\tau]=0.
\end{eqnarray}
 Let $\displaystyle v^{(1)}, v^{(2)}$ be the functions which correspond to $\displaystyle L^{(1)}, L^{(2)}\in \Sigma(A,T)$. We denote $\displaystyle K[L^{(i)}]=\sigma^{(i)}$, $i=1,2$, we apply  $(\ref{34512})$ and  let $\epsilon\rightarrow 0$, $R\rightarrow \infty$. We get
\begin{center}
$\displaystyle \int_{0}^{t}2e^{-\tau}L_\tau^{(1)}\left[\frac{d}{d\tau}(\sigma_{\tau}^{(1)}-\sigma_{\tau}^{(2)})\right]d\tau+\int_{0}^{t}2e^{-\tau}\left[\frac{d}{d\tau}\sigma_{\tau}^{(2)}\right](L_{\tau}^{(1)}-L_{\tau}^{(2)})d\tau$\\
$\displaystyle=\int_{L_{t}^{(1)}}^{\infty}\xi v^{(1)}(\xi,t)d\xi-\int_{L_{t}^{(2)}}^{\infty}\xi v^{(2)}(\xi,t)d\xi.$
\end{center}
By integration by parts  we have
\begin{eqnarray*}
\displaystyle 2e^{-t}L_t^{(1)}[\sigma_t^{(1)}-\sigma_t^{(2)}]= \int_{0}^{t}(-2e^{-\tau}L_\tau^{(1)}+2e^{-\tau}\frac{d}{d\tau}L_\tau^{(1)})(\sigma_{\tau}^{(1)}-\sigma_{\tau}^{(2)})d\tau
\end{eqnarray*}
\begin{eqnarray}
\label{3424141}
-\int_{0}^{t}2e^{-\tau}\left[\frac{d}{d\tau}\sigma_{\tau}^{(2)}\right](L_{\tau}^{(1)}-L_{\tau}^{(2)})d\tau+\left[\int_{L_{t}^{(1)}}^{\infty}\xi v^{(1)}(\xi,t)d\xi-\int_{L_{t}^{(2)}}^{\infty}\xi v^{(2)}(\xi,t)d\xi\right].
\end{eqnarray}
\vskip 0.3cm
To control $\displaystyle \int_{L_{t}^{(1)}}^{\infty}\xi v^{(1)}(\xi,t)d\xi-\int_{L_{t}^{(2)}}^{\infty}\xi v^{(2)}(\xi,t)d\xi$, using $(\ref{777})$ and Fubini's theorem, we obtain for $i=1,2$:
\begin{eqnarray*}
\displaystyle\int_{L_t^{(i)}}^{\infty}xv^{(i)}(x,t)dx=\int_{b}^{\infty}h(\xi)\int_{L_t^{(i)}}^{\infty}xG(x,t;\xi,0)dxd\xi-\frac{1}{2}\int_{0}^{t}v_{\xi}^{(i)}(L_\tau^{(i)},\tau)\int_{L_t^{(i)}}^{\infty}xG(x,t;L_\tau^{(i)},\tau)dxd\tau
\end{eqnarray*}
\begin{eqnarray}
\label{44123}
+\int_{0}^{t}e^{-\tau}\int_{L_{t}^{(i)}}^{\infty}xG_\xi(x,t;L_\tau^{(i)},\tau)dxd\tau.
\end{eqnarray}
\vskip 0.3cm
Taking the difference for the first term of $(\ref{44123})$ , we have
\begin{eqnarray*}
I_1:=\int_{b}^{\infty}h(\xi)t[G(L_t^{(1)},t;\xi,0)-G(L_t^{(2)},t;\xi,0)]d\xi+\int_{L_t^{(1)}}^{L_t^{(2)}}\int_{b}^{\infty}\xi h(\xi)G(x,t;\xi,0)d\xi dx.
\end{eqnarray*}
so that $\displaystyle \left|I_1\right|\leq C\sup_{0\leq\tau\leq T}\left|L_{\tau}^{(1)}-L_{\tau}^{(2)}\right|$.
\vskip 0.3cm
Since
\begin{eqnarray*}
-\frac{1}{2}\int_{0}^{t}v_\xi^{(i)}(L_\tau^{(i)},\tau)\int_{L_t^{(i)}}^{\infty}xG(x,t;L_\tau^{(i)},\tau)dxd\tau=2\int_{0}^{t}e^{-\tau}\left[\frac{d}{d\tau}\sigma_{\tau}^{(i)}\right]\int_{L_t^{(i)}}^{\infty}xG(x,t;L_\tau^{(i)},\tau)dxd\tau,
\end{eqnarray*}
we also have by taking the difference for the second term of $(\ref{44123})$,
\begin{eqnarray*}
I_2:=2\int_{0}^{t}e^{-\tau}\frac{d}{d\tau}(\sigma_{\tau}^{(1)}-\sigma_{\tau}^{(2)})\int_{L_t^{(1)}}^{\infty}xG(x,t;L_\tau^{(1)},\tau)dxd\tau
\end{eqnarray*}
\begin{eqnarray}
\label{4552}
+2\int_{0}^{t}\left[\int_{L_t^{(1)}}^{\infty}xG(x,t;L_\tau^{(1)},\tau)dx-\int_{L_t^{(2)}}^{\infty}xG(x,t;L_\tau^{(2)},\tau)dx\right]e^{-\tau}\frac{d}{d\tau}\sigma_{\tau}^{(2)}d\tau.
\end{eqnarray}

By integration by parts, we get for the first term on right hand side of (\ref{4552})
\begin{eqnarray*}
&&2\int_{0}^{t}e^{-\tau}\frac{d}{d\tau}(\sigma_{\tau}^{(1)}-\sigma_{\tau}^{(2)})\int_{L_t^{(1)}}^{\infty}xG(x,t;L_\tau^{(1)},\tau)dxd\tau\\
\\&&\hskip1cm=e^{-t}[\sigma_{t}^{(1)}-\sigma_{t}^{(2)}]L_t^{(1)}-2\int_{0}^{t}[\sigma_\tau^{(1)}-\sigma_\tau^{(2)}]\frac{\partial}{\partial\tau}\left(e^{-\tau}\int_{L_t^{(1)}}^{\infty}xG(x,t;L_\tau^{(1)},\tau)dx\right)d\tau.
\end{eqnarray*}
so that
\begin{eqnarray*}
I_2=e^{-t}[\sigma_{t}^{(1)}-\sigma_{t}^{(2)}]L_t^{(1)}-2\int_{0}^{t}[\sigma_\tau^{(1)}-\sigma_\tau^{(2)}]\frac{\partial}{\partial\tau}\left(e^{-\tau}\int_{L_t^{(1)}}^{\infty}xG(x,t;L_\tau^{(1)},\tau)dx\right)d\tau
\end{eqnarray*}
\begin{eqnarray}
\label{3414141}
+2\int_{0}^{t}\left[\int_{L_t^{(1)}}^{\infty}xG(x,t;L_\tau^{(1)},\tau)dx-\int_{L_t^{(2)}}^{\infty}xG(x,t;L_\tau^{(2)},\tau)dx\right]e^{-\tau}\frac{d}{d\tau}\sigma_{\tau}^{(2)}d\tau.
\end{eqnarray}

To estimate the last term of $\eqref{3414141}$, we use the following identity for i=1,2;
\begin{eqnarray*}
\int_{L_t^{(i)}}^{\infty}xG(x,t;L_\tau^{(i)},\tau)dx=-(t-\tau)\int_{L_t^{(i)}}^{\infty}G_x(x,t;L_\tau^{(i)},\tau)dx+L_{\tau}^{(i)}\int_{L_t^{(i)}}^{\infty}G(x,t;L_\tau^{(i)},\tau)dx
\end{eqnarray*}
\begin{eqnarray*}
=(t-\tau)G(L_t^{(i)},t;L_\tau^{(i)},\tau)+L_{\tau}^{(i)}\int_{L_t^{(i)}}^{\infty}G(x,t;L_\tau^{(i)},\tau)dx=(t-\tau)G(L_t^{(i)},t;L_\tau^{(i)},\tau)+L_{\tau}^{(i)}\Psi\left(\frac{L_t^{(i)}-L_\tau^{(i)}}{\sqrt{t-\tau}}\right),
\end{eqnarray*}
where $\displaystyle \Psi(x)=\int_{x}^{\infty}\frac{1}{\sqrt{2\pi}}\exp\left\{{-\frac{|z|^2}{2}}\right\}dz$.
\vskip 0.3cm
By applying the mean value theorem on $\Psi$, we have
\begin{eqnarray*}
\displaystyle\left| \Psi\left(\frac{L_t^{(1)}-L_\tau^{(1)}}{\sqrt{t-\tau}} \right) - \Psi\left(\frac{L_t^{(2)}-L_\tau^{(2)}}{\sqrt{t-\tau}} \right)   \right|\leq C_1\left|\frac{L_t^{(1)}-L_\tau^{(1)}-L_t^{(2)}+L_\tau^{(2)} }{\sqrt{t-\tau}}\right|\leq\frac{C_1\sup_{0\leq\eta\leq T}\left|L_{\eta}^{(1)}-L_{\eta}^{(2)}\right|}{\sqrt{t-\tau}}
\end{eqnarray*}

By applying the mean value theorem on $\displaystyle \exp\left\{{-\frac{|\cdot|^2}{2}}\right\}$, for some $0<\theta<1$,
\begin{eqnarray*}
&&|G(L_t^{(1)},t;L_\tau^{(1)},\tau)-G(L_t^{(2)},t;L_\tau^{(2)},\tau)|
\\&&\hskip 1cm\leq\frac{1}{\sqrt{2\pi(t-\tau)}}\left|\theta\frac{L_t^{(1)}-L_\tau^{(1)}}{\sqrt{t-\tau}}+(1-\theta)\frac{L_t^{(2)}-L_\tau^{(2)}}{\sqrt{t-\tau}}  \right| \left|\frac{L_t^{(1)}-L_\tau^{(1)}-L_t^{(2)}+L_\tau^{(2)} }{\sqrt{t-\tau}}\right|\\&&\hskip 1cm
\leq\frac{C_2\sup_{0\leq\eta\leq T}\left|L_{\eta}^{(1)}-L_{\eta}^{(2)}\right|}{\sqrt{t-\tau}}
\end{eqnarray*}
Thus we obtain
\begin{eqnarray*}
\displaystyle\left|\int_{L_t^{(1)}}^{\infty}xG(x,t;L_\tau^{(1)},\tau)dx-\int_{L_t^{(2)}}^{\infty}xG(x,t;L_\tau^{(2)},\tau)dx \right|\leq C_3\sup_{0\leq\eta\leq T}\left|L_{\eta}^{(1)}-L_{\eta}^{(2)}\right|+C_4\frac{\sup_{0\leq\eta\leq T}\left|L_{\eta}^{(1)}-L_{\eta}^{(2)}\right|}{\sqrt{t-\tau}}.
\end{eqnarray*}
Also we observe
\begin{eqnarray*}
\left|\frac{\partial}{\partial\tau}\left(e^{-\tau}\int_{L_t^{(1)}}^{\infty}xG(x,t;L_\tau^{(1)},\tau)dx\right)d\tau\right|
\end{eqnarray*}
\begin{eqnarray*}
=\left|\frac{\partial}{\partial\tau}\left[e^{-\tau}\left\{(t-\tau)G(L_t^{(1)},t;L_\tau^{(1)},\tau)+L_\tau^{(1)}\Psi\left(\frac{L_t^{(1)}-L_\tau^{(1)}}{\sqrt{t-\tau}}\right)\right\}\right] \right|
\leq C_5+\frac{C_6}{\sqrt{t-\tau}}.
\end{eqnarray*}

For the third term of $(\ref{44123})$ we have
\begin{eqnarray*}
\int_{0}^{t}e^{-\tau}\int_{L_{t}^{(i)}}^{\infty}xG_\xi(x,t;L_\tau^{(i)},\tau)dxd\tau=\int_{0}^{t}e^{-\tau}\left[L_t^{(i)}G(L_t^{(i)},t;L_\tau^{(i)},\tau)+\int_{L_t^{(i)}}^{\infty}G(x,t;L_\tau^{(i)},\tau)\right]d\tau.
\end{eqnarray*}
Then we have
\begin{eqnarray*}
I_3:=\int_{0}^{t}e^{-\tau}\left[L_t^{(1)}G(L_t^{(1)},t;L_\tau^{(1)},\tau)+\int_{L_t^{(1)}}^{\infty}G(x,t;L_\tau^{(1)},\tau)-L_t^{(2)}G(L_t^{(2)},t;L_\tau^{(2)},\tau)-\int_{L_t^{(2)}}^{\infty}G(x,t;L_\tau^{(2)},\tau)\right]d\tau
\end{eqnarray*}
so that
\begin{eqnarray*}
|I_3|\leq C_7\sup_{0\leq\eta\leq T}\left|L_{\eta}^{(1)}-L_{\eta}^{(2)}\right|.
\end{eqnarray*}

Combining all the results and $(\ref{3424141})$ finally we obtain for all sufficiently small $T>0$,
\begin{eqnarray*}
\displaystyle \left|\sigma_t^{(1)}-\sigma_t^{(2)}\right|\leq C_{8}\int_{0}^{t}\left|\sigma_\tau^{(1)}-\sigma_\tau^{(2)}\right|d\tau+C_{9}\int_{0}^{t}\frac{\left|\sigma_\tau^{(1)}-\sigma_\tau^{(2)}\right|}{\sqrt{t-\tau}}d\tau+C_{10}\sup_{0\leq\tau\leq T}\left|L_{\tau}^{(1)}-L_{\tau}^{(2)}\right|\\
\leq (C_{8}+C_{9})\int_{0}^{t}\frac{\left|\sigma_\tau^{(1)}-\sigma_\tau^{(2)}\right|}{\sqrt{t-\tau}}d\tau+C_{10}\sup_{0\leq\tau\leq T}\left|L_{\tau}^{(1)}-L_{\tau}^{(2)}\right|.
\end{eqnarray*}

By  Lemma 17.7.1   in \cite{Cannon}, we get $\displaystyle \sup_{0\leq\tau\leq T}\left|\sigma_t^{(1)}-\sigma_t^{(2)}\right|\leq C_{11}\sup_{0\leq\tau\leq T}\left|L_{\tau}^{(1)}-L_{\tau}^{(2)}\right|$ so that $K$ is a continuous map.
\end{proof}

\begin{proof}{\textbf{of Theorem \ref{3424}}}
\\Let $(L,v)$ as in Lemma \ref{637} then by (\ref{40303})  $K[L]=L$ is in $\displaystyle C^{1}([0,T])$. This completes the proof.
\end{proof}

\section{Proof of Theorem \ref{34}}

Suppose that $h(x) = \dfrac{d\rho_0(x)}{dx}$ and let
$(L,v)$ be as in Theorem \ref{3424} with such $h$ as initial condition.
 Let $\displaystyle\rho(x,t):=\int_{L_t}^{x}e^{t}v(y,t)dy$. It can be readily shown that $\rho$ solves from the first  to the third equations in $(*)$, see Section \ref{sec:intro}. To prove the last equation in $(*)$  we first
 remark that by \eqref{???}
 $v\in L^1([L_t,\infty))$.
 % by using () and Fubini's theorem, then we have
% \begin{eqnarray}
% \displaystyle \int_{L_t}^{\infty}v(y,t)dy=\int_{b}^{\infty}h(\xi)\int_{L_t}^{\infty}G(y,t;\xi,0)dyd\xi-\int_{0}^{t}G(L_t,t;L_\tau,\tau)\varphi(\tau)d\tau.
% \end{eqnarray}
We then differentiate $\displaystyle \int_{L_t}^{\infty}v(y,t)dy$ with respect to $t$:
 \begin{eqnarray*}
 \displaystyle \frac{d}{dt}\left(\int_{L_t}^{\infty}v(x,t)dx\right)=-\dot{L_t}v(L_t,t)+\int_{L_t}^{\infty}v_t(x,t)dx\\=-2\dot{L_t}e^{-t}+\int_{L_t}^{\infty}\frac{1}{2}v_{xx}(x,t)
 =-2\dot{L_t}e^{-t}-\frac{1}{2}v_x(L_t,t)=0.
 \end{eqnarray*}
 Since $\displaystyle \int_{b}^{\infty}h(x)dx=\int_{b}^{\infty}\frac{d\rho_0(x)}{dx}dx=0$, we obtain that $\displaystyle \int_{L_t}^{\infty}v(x,t)dx=0$. By using this, $(\ref{???})$, and Fubini's theorem, we also have
  \begin{eqnarray}
\rho(x,t)=e^{t}\left[\int_{0}^{t}G(x,t;L_\tau,\tau)\varphi(\tau)d\tau-\int_{b}^{\infty}h(\xi)\int_{x}^{\infty}G(y,t;\xi,0)dyd\xi\right]
 \end{eqnarray}
so that $\displaystyle\int_{L_t}^{\infty}\rho(x,t)dx<\infty$. Similarly, if we differentiate $\displaystyle\int_{L_t}^{\infty}\rho(x,t)dx$ with respect to $t$ and get $\displaystyle \frac{d}{dt}\left(\int_{L_t}^{\infty}\rho(x,t)dx\right)=\int_{L_t}^{\infty}\rho_t(x,t)dx=\int_{L_t}^{\infty}\left(\frac{1}{2}\rho_{xx}(x,t)+\rho(x,t)\right)dx=-\frac{1}{2}\rho_x(L_t,t)+\int_{L_t}^{\infty}\rho(x,t)dx=-1+\int_{L_t}^{\infty}\rho(x,t)dx$. Since $\displaystyle \int_{b}^{\infty}\rho_0(x)dx=1$, we obtain that $\displaystyle \int_{L_t}^{\infty}\rho(x,t)dx=1$. The derivative $\rho_x(x,t)$ has a limit when $x\rightarrow L_t$ and $\displaystyle\rho_x(L_t,t)=2$ and vanishes as $x\to \infty$ by Proposition \ref{96}. This completes the proof of  Theorem \ref{34}; the statements below Theorem \ref{34} also follow from what proved for $v(x,t)$.

\vskip2cm

{\bf Acknowledgments.}
I thank A. De Masi, P. Ferrari and E. Presutti for useful discussions.


\begin{thebibliography}{}

\bibitem{Cannon} J.R. Cannon: \textit{The One-Dimensional Heat Equation}, Addison-Wesley Publishing Company (1984).


\bibitem{Fasano} A. Fasano, \textit{Mathematical models of some diffusive processes with free boundaries}, SIMAI e-Lecture Notes (2008).

\bibitem{BD} E.Brunet, B.Derrida, \textit{Shift in the velocity of a front due to a cutoff}, Phys. Rev. E (3) 56 2597-2604. MR1473413 (1997).

\bibitem{BH} C.M.Brauner, J.Hulshof,  \textit{A General Approach to Stability in Free Boundary Problems}, Journal of Differential Equations 164, 1648 (2000).

\bibitem{CV} L.A. Caffarelli, J.L.Vazquez, \textit{A free-boundary problem for the heat equation arising in flame propagation}, Transactions of the American Mathematical Society, Volume 347, Number 2, February (1995).

\bibitem{CDGP}G.Carinci, A. De Masi, C.Giardin\`{a}, E.Presutti,  \textit{Free Boundary Problems in PDEs and Particle Systems}, Springer (2016).

\bibitem{DFPS} A. De Masi, P.Ferrari, E.Presutti, N.Soprano-Loto,  \textit{Hydrodynamics of the NBBM process}, Preprint, https://arxiv.org/abs/1707.00799 (2017).

\bibitem{BBD} J.Berestycki, E.Brunet, B.Derrida,  \textit{Exact solution and precise asymptotics of a Fisher-KPP type front}, Preprint, https://arxiv.org/abs/1705.08416 (2017).

%\bibitem{Schwierz} F. Schwierz, \textit{Graphene transistors}, Nature Nanotechnology \textbf{5}, 487--496 (2010).

%\bibitem{Shahil} K. M. F. Shahil and A. A. Balandin, \textit{Graphene?Multilayer Graphene Nanocomposites as Highly Efficient Thermal Interface Materials}, Nano Lett. \textbf{12}, 861--867 (2012).

%\bibitem{Yanko} M. Yankowitz,	 J. I-J. Wang,	A. G. Birdwell,	Yu-An Chen, K. Watanabe,T. Taniguchi, P. Jacquod,	P. San-Jose, P. Jarillo-Herrero and B. J. LeRoy, \textit{Electric field control of soliton motion and stacking in trilayer graphene},
%Nature Materials \textbf{13}, 786--789 (2014).

\end{thebibliography}
\end{document}